%% file: HDDifferenceSets.tex
\documentclass{amsart}
\usepackage[mathscr]{eucal}

\usepackage{hyperref}

\allowdisplaybreaks[1]

\usepackage{relsize}

\PassOptionsToPackage{usenames,dvipsnames,svgnames}{xcolor}  
\usepackage{tikz}
\usetikzlibrary{arrows,positioning,automata}
\usetikzlibrary{shapes,snakes} 

\newtheorem{thrm}{Theorem}[section]
\newtheorem{lem}[thrm]{Lemma}

\newtheorem{cor}[thrm]{Corollary}
\theoremstyle{definition}
\newtheorem{definition}[thrm]{Definition}

\numberwithin{equation}{section}

\newtheorem{example}[thrm]{Example}
\newtheorem{problem}[thrm]{Problem}

\newcommand{\labeq}[1]{\label{eq:#1}}
\newcommand{\refeq}[1]{(\ref{eq:#1})}
\newcommand{\labt}[1]{\label{thrm:#1}}
\newcommand{\reft}[1]{Theorem~\ref{thrm:#1}}

\newcommand{\labl}[1]{\label{lemma:#1}}
\newcommand{\refl}[1]{Lemma~\ref{lemma:#1}}
\newcommand{\labd}[1]{\label{definition:#1}}
\newcommand{\refd}[1]{Definition~\ref{definition:#1}}

\newcommand{\refc}[1]{Corollary~\ref{coro:#1}}

\newcommand{\labf}[1]{\label{fig:#1}}
\newcommand{\reff}[1]{Figure~\ref{fig:#1}}

\newcommand{\dimm}[1]{\hbox{$\dim_{\hbox{M}}$}\left( #1\right)}
\newcommand{\dimh}[1]{\hbox{$\dim_{\hbox{H}}$}\left( #1\right)}

\newcommand{\lmeas}[1]{\lambda\left( #1 \right)}

\newcommand{\NN}{\mathbb{N}_2^{\mathbb{N}}}

\newcommand{\wrt}[1]{\hbox{ w.r.t. }#1}
\newcommand{\wrtQ}{\hbox{ w.r.t. }Q}

\newcommand{\floor}[1]{\left\lfloor #1 \right\rfloor} 
\newcommand{\ceil}[1]{\left\lceil #1 \right\rceil} 
\newcommand{\br}[1]{\left\{ #1 \right\}}

\newcommand{\pr}[1]{\left( #1 \right)}

\newcommand{\NQ}{\mathscr{N}(Q)}
\newcommand{\N}[1]{\mathscr{N}( #1 )}
\newcommand{\Nk}[2]{\mathscr{N}_{#2}( #1 )} 
\newcommand{\DNQ}{\mathscr{DN}(Q)}
\newcommand{\DN}[1]{\mathscr{DN}( #1 )} 
\newcommand{\RNQ}{\mathscr{RN}(Q)}
\newcommand{\RN}[1]{\mathscr{RN}( #1 )} 
\newcommand{\RNk}[2]{\mathscr{RN}_{#2}( #1 )} 
\newcommand{\RDN}{\RNQ \cap \DNQ \backslash \NQ}

\newcommand{\qnk}{Q_n^{(k)}}

\newcommand{\pnk}{P_n^{(k)}}

\newcommand{\SQx}{\mathscr{S}_Q(x)}
\newcommand{\WQS}{\mathscr{W}_Q(S)}

\newcommand{\blank}[1]{ }

\thanks{Research of the authors is partially supported by the U.S. NSF grant DMS-0943870. }

\author[D. Airey]{Dylan Airey}
\address[D. Airey]{
Department of Mathematics, University of Texas at Austin, 2515 Speedway, Austin, TX 78712-1202, USA}
\email{dylan.airey@utexas.edu}

\author[B. Mance]{Bill Mance}
\address[B. Mance]{Department of Mathematics, University of North Texas, General Academics Building 435, 1155 Union Circle,  \#311430, Denton, TX 76203-5017, USA}
\email{mance@unt.edu}

\begin{document}

\title[Hausdorff dimension of digit sets]{On the Hausdorff dimension of some sets of numbers defined through the digits of their $Q$-Cantor series expansions}

\begin{abstract}
Following in the footsteps of P. Erd\H{o}s, A. R\'{e}nyi, and T. \u{S}al\'{a}t we compute the Hausdorff dimension of sets of numbers whose digits with respect to their $Q$-Cantor series expansions satisfy various statistical properties.  In particular, we consider difference sets associated with various notions of normality and sets of numbers with a prescribed range of digits.
\end{abstract}

\maketitle

\section{Introduction}

The study of normal numbers and other statistical properties of real numbers with respect to large classes of Cantor series expansions was  first done by P. Erd\H{o}s and A. R\'{e}nyi in \cite{ErdosRenyiConvergent} and \cite{ErdosRenyiFurther} and by A. R\'{e}nyi in \cite{RenyiProbability}, \cite{Renyi}, and \cite{RenyiSurvey} and by P. Tur\'{a}n in \cite{Turan}.

The $Q$-Cantor series expansions, first studied by G. Cantor in \cite{Cantor},
are a natural generalization of the $b$-ary expansions.\footnote{G. Cantor's motivation to study the Cantor series expansions was to extend the well known proof of the irrationality of the number $e=\sum 1/n!$ to a larger class of numbers.  Results along these lines may be found in the monograph of J. Galambos \cite{Galambos}. } 
Let $\mathbb{N}_k:=\mathbb{Z} \cap [k,\infty)$.  If $Q \in \NN$, then we say that $Q$ is a {\it basic sequence}.
Given a basic sequence $Q=(q_n)_{n=1}^{\infty}$, the {\it $Q$-Cantor series expansion} of a real number $x$  is the (unique)\footnote{Uniqueness can be proven in the same way as for the $b$-ary expansions.} expansion of the form
\begin{equation} \labeq{cseries}
x=E_0+\sum_{n=1}^{\infty} \frac {E_n} {q_1 q_2 \cdots q_n}
\end{equation}
where $E_0=\floor{x}$ and $E_n$ is in $\{0,1,\ldots,q_n-1\}$ for $n\geq 1$ with $E_n \neq q_n-1$ infinitely often. We abbreviate \refeq{cseries} with the notation $x=E_0.E_1E_2E_3\ldots$ w.r.t. $Q$.


A {\it block} is an ordered tuple of non-negative integers, a {\it block of length $k$} is an ordered $k$-tuple of integers, and {\it block of length $k$ in base $b$} is an ordered $k$-tuple of integers in $\{0,1,\ldots,b-1\}$.

Let
$$
Q_n^{(k)}:=\sum_{j=1}^n \frac {1} {q_j q_{j+1} \cdots q_{j+k-1}} \hbox{ and }  T_{Q,n}(x):=\left(\prod_{j=1}^n q_j\right) x \pmod{1}.
$$
A. R\'enyi \cite{Renyi} defined a real number $x$ to be {\it normal} with respect to $Q$ if for all blocks $B$ of length $1$,
\begin{equation}\labeq{rnormal}
\lim_{n \rightarrow \infty} \frac {N_n^Q (B,x)} {Q_n^{(1)}}=1.
\end{equation}
If $q_n=b$ for all $n$ and we restrict $B$ to consist of only digits less than $b$, then \refeq{rnormal} is equivalent to {\it simple normality in base $b$}, but not equivalent to {\it normality in base $b$}. 
A basic sequence $Q$ is {\it $k$-divergent} if
$\lim_{n \rightarrow \infty} Q_n^{(k)}=\infty$,  {\it fully divergent} if $Q$ is $k$-divergent for all $k$, and {\it $k$-convergent} if it is not $k$-divergent.  
A basic sequence $Q$ is {\it infinite in limit} if $q_n \rightarrow \infty$.

\begin{definition}\labd{1.7} A real number $x$  is {\it $Q$-normal of order $k$} if for all blocks $B$ of length $k$,
$$
\lim_{n \rightarrow \infty} \frac {N_n^Q (B,x)} {Q_n^{(k)}}=1.
$$
We let $\Nk{Q}{k}$ be the set of numbers that are $Q$-normal of order $k$.  The real number $x$ is {\it $Q$-normal} if
$x \in \NQ := \bigcap_{k=1}^{\infty} \Nk{Q}{k}.$
$x$ is {\it $Q$-ratio normal of order $k$} (here we write $x \in \RNk{Q}{k}$) if for all blocks $B_1$ and $B_2$ of length $k$
$$
\lim_{n \to \infty} \frac {N_n^Q (B_1,x)} {N_n^Q (B_2,x)}=1.
$$
$x$ is {\it $Q$-ratio normal} if
$
x \in \RNQ := \bigcap_{k=1}^{\infty} \RNk{Q}{k}.
$
A real number~$x$ is {\it $Q$-distribution normal} if
the sequence $(T_{Q,n}(x))_{n=0}^\infty$ is uniformly distributed mod $1$.  Let $\DNQ$ be the set of $Q$-distribution normal numbers.
\end{definition}

It was proven in \cite{ppq1} that the directed graph in \reff{figure1} gives the complete containment relationships between these notions when $Q$ is infinite in limit and fully divergent.  The vertices are labeled with all possible intersections of one, two, or three choices of the sets $\NQ$, $\RNQ$, and $\DNQ$.  The set labeled on vertex $A$ is a subset of the set labeled on vertex $B$ if and only if there is a directed path from $A$ to $B$. 
For example, $\NQ \cap \DNQ \subseteq \RNQ$, so all numbers that are $Q$-normal and $Q$-distribution normal are also $Q$-ratio normal.  

Note that in base~$b$, where $q_n=b$ for all $n$,
 the corresponding notions of $Q$-normality, $Q$-ratio normality, and $Q$-distribution normality are equivalent. This equivalence
is fundamental in the study of normality in base $b$.

\begin{figure}
\caption{}
\labf{figure1}
\begin{tikzpicture}[>=stealth',shorten >=1pt,node distance=3.8cm,on grid,initial/.style    ={}]
  \node[state]          (NQ)                        {$\mathsmaller{\NQ}$};
  \node[state]          (RNQ) [left =of NQ]    {$\mathsmaller{\RNQ}$};
  \node[state]          (NQRNQ) [above right =of NQ]    {$\mathsmaller{\NQ \cap \RNQ}$};
  \node[state]          (RNQDNQ) [above left=of RNQ]    {$\mathsmaller{\RNQ \cap \DNQ}$};
  \node[state]          (NQDNQ) [above left =of NQ]    {$\mathsmaller{\NQ \cap \DNQ}$};
  \node[state]          (NQRNQDNQ) [above right =of NQDNQ]    {$\mathsmaller{\NQ \cap \RNQ \cap \DNQ}$};
  \node[state]          (DNQ) [above right=of RNQDNQ]    {$\mathsmaller{\DNQ}$};
\tikzset{mystyle/.style={->,double=black}} 
\tikzset{every node/.style={fill=white}} 
\path (RNQDNQ)     edge [mystyle]    (RNQ)
      (RNQDNQ)     edge [mystyle]     (DNQ)
      (NQ)     edge [mystyle]     (RNQ)
      (NQDNQ)     edge [mystyle]     (RNQDNQ)
      (NQDNQ)     edge [mystyle]     (NQ);
\tikzset{mystyle/.style={<->,double=black}}
\path (NQRNQDNQ)     edge [mystyle]    (NQDNQ)
	(NQ)     edge [mystyle]    (NQRNQ);
\end{tikzpicture}
\end{figure}

It  follows from a well known result of H. Weyl \cite{Weyl2,Weyl4} that $\DNQ$ is a set of full Lebesgue measure for every basic sequence $Q$. We will need the following result of the second author \cite{Mance4} later in this paper.

\begin{thrm}\labt{measure}\footnote{Early work in this direction has been done by A. R\'enyi \cite{Renyi}, T.  \u{S}al\'at \cite{Salat4}, and F. Schweiger~\cite{SchweigerCantor}.}
Suppose that $Q$ that is infinite in limit.  Then $\Nk{Q}{k}$ and $\RNk{Q}{k}$ are of full measure if and only if $Q$ is $k$-divergent. The sets $\NQ$ and $\RNQ$ are of full measure if and only if $Q$ is fully divergent. 
\end{thrm}

Based on \reff{figure1} and \reft{measure} it is natural to ask for the Hausdorff dimension of the difference sets.  
It was proven in \cite{Mance7} that for every basic sequence $Q$ that is infinite in limit
$$
\dimh{\DNQ \backslash \NQ}=\dimh{\DNQ \backslash \RNQ}=1.
$$
Using different methods we will prove the following theorem.
\begin{thrm}
Every non-empty difference set expressed in terms of $\NQ$, $\RNQ$, and $\DNQ$, possibly involving intersections and unions, has full Hausdorff dimension for every $Q$ that is infinite in limit, except for the set $\NQ \backslash \DNQ$, 
\end{thrm}
It will be shown that the set $\NQ \backslash \DNQ$ has full Hausdorff dimension for a more restricted class of basic sequences in \reft{NDN}.  We should note that we can not hope to establish $\dimh{\NQ \backslash \DNQ}=1$ for all $Q$ that are infinite in limit.  This follows from the result in \cite{Mance4} that $\NQ=\emptyset$ when $Q$ is infinite in limit and not fully divergent.

A surprising property of $Q$-normality of order $k$ is that we may not conclude that $\Nk{Q}{k} \subseteq \Nk{Q}{j}$ for all $j <k$ like we may for the $b$-ary expansions.  In fact, it was shown in \cite{ppq2} that for every $k$ there exists a basic sequence $Q$ and a real number $x$ such that $\Nk{Q}{k} \backslash \bigcup_{j=1}^{k-1} \Nk{Q}{j}$ is non-empty.  
Thus, we will have to be more careful in stating exactly what our theorems prove since lack of $Q$-normality of order $2$ does not imply lack of $Q$-normality of order $338$, for example.  Furthermore, we will greatly expand on this result in \reft{NDNk} where for each natural number $\ell$ we exhibit a class of basic sequences such that
$$
\dimh{\bigcap_{j=\ell}^\infty \Nk{Q}{j} \Big \backslash \bigcup_{j=1}^{\ell-1} \Nk{Q}{j}}=1.
$$
For $x=E_0.E_1E_2\cdots\wrt{Q}$, define the set
$$
\SQx=\{E_1,E_2,E_3,\cdots\}.
$$
P. Erd\H{o}s  and A. R\'{e}nyi \cite{ErdosRenyiConvergent} proved the following theorems.

\begin{thrm}[P. Erd\H{o}s  and A. R\'{e}nyi]\labt{erdosrenyidensity}
If $Q$ is $1$-convergent, then $\SQx$ has density $0$ for almost every real number $x$.
\end{thrm}
\begin{thrm}[P. Erd\H{o}s  and A. R\'{e}nyi]\labt{erdosrenyinumberdigits}
For $x=E_0.E_1E_2\cdots\wrtQ$, let $d_n(x)$ denote the number of different numbers in the sequence $E_1,\cdots,E_n$.  If $Q$ is $1$-convergent, then  for almost every $x$ we have
$
\lim_{n \to \infty} \frac {d_n(x)}{n}=1.
$
\end{thrm}

It should be noted that T. \u{S}al\'{a}t \cite{Salat5} 
considered sets related to those mentioned in \reft{erdosrenyidensity} and \reft{erdosrenyinumberdigits}.
We will need the following definition from \cite{BarlowTaylorMassDimension}.

\begin{definition}\labd{massdimension}
For $S \subseteq \mathbb{Z}$, define the {\it mass dimension of $S$} to be the limit
$$
\dimm{S}=\lim_{n \to \infty} \frac {\log \#(S \cap (-n/2,n/2))}{\log n},
$$
if it exists.
\end{definition}

We note that an {\it upper mass dimension} and a {\it lower mass dimension} may be defined similarly by changing the limit in \refd{massdimension} to a $\limsup$ or a $\liminf$.

For non-empty $S \subseteq \mathbb{N}_0$, define
$$
\WQS=\br{x \in \mathbb{R} : \SQx=S}.
$$

We will build on \reft{erdosrenyidensity} and \reft{erdosrenyinumberdigits} by proving the following theorem.
\begin{thrm} \labt{digitrange}
If $Q$ is infinite in limit, $\lim_{n \to \infty} \frac{\log q_n}{\sum_{i=1}^n \log q_i} =0$, and $S\subseteq \mathbb{N}$ such that $\min S<\min Q$ and $\dimm{S}$ exists, then 
$$
\dimh{\WQS} = \dimm{S}.
$$
\end{thrm}
T. \u{S}al\'{a}t proved in \cite{Salat3} that under some conditions on the basic sequence $Q$ the set of real numbers whose digits in their $Q$-Cantor expansion is bounded has zero Hausdorff dimension.  We remark that his result may be sharpened with his conditions weakened by use of our \refl{HDT} instead of Satz 1 from \cite{Salat7}.  The proof of this otherwise follows identically to his original proof, so we do not record it in this paper.

If $Q$ is infinite in limit and not fully divergent, then $\lmeas{\RNQ}=0$.  We will show as a consequence of the following theorem that $\dimh{\RNQ}=1$ whenever $Q$ is infinite in limit.
\begin{thrm}\labt{RDN}
If $Q$ is infinite in limit, then $\dimh{\RDN}=1$.
\end{thrm}

Lastly, we remark that some of the techniques developed in this paper and \refl{HDT} are used to study fractals associated with normality preserving operations in \cite{AireyManceVandehey}.

\section{Lemmas}
 Let $( n_k )$ be a sequence of positive integers and $(c_k)$ be a sequence of positive numbers such that $n_k \geq 2$, $0<c_k<1$, $n_1 c_1 \leq \delta$, and $n_k c_k \leq 1$, where $\delta$ is a positive real number. For any $k$, let $D_k = \{ (i_1, \cdots, i_k): 1\leq i_j \leq n_j, 1\leq j \leq k \}$, and $D = \bigcup D_k$, where $D_0 =\emptyset$. If $\sigma = ( \sigma_1, \cdots , \sigma_k) \in D_k$, $\tau = (\tau_1 ,\cdots , \tau_m) \in D_m$, put $\sigma * \tau = (\sigma_1, \cdots , \sigma_k, \tau_1, \cdots , \tau_m)$.

\begin{definition}
Suppose $J$ is a closed interval of length $\delta$.  The collection of closed subintervals $ \mathcal{F} = \{ J_\sigma : \sigma \in D\}$ of $J$ has \textit{homogeneous Moran structure} if:
\begin{enumerate}
	\item $J_{\emptyset} = J$;
	\item $\forall k \geq 0, \sigma \in D_k, J_{\sigma *1}, \cdots , J_{\sigma * n_{k+1}}$ are subintervals of $J_\sigma$ and $\mathring{J}_{\sigma*i}\cap \mathring{J}_{\sigma*j}=\emptyset$ for $i \neq j$;
	\item $\forall k \geq 1, \forall \sigma \in D_{k-1}, 1\leq j \leq n_k$, $c_k = \frac{\lambda(J_{\sigma*j})}{\lambda(J_\sigma)}$.
\end{enumerate}
\end{definition}

Suppose that $\mathcal{F}$ is a collection of closed subintervals of $J$ having homogeneous Moran structure. Let $E(\mathcal{F}) = \bigcap_{k\geq 1} \bigcup_{\sigma \in D_k} J_\sigma$. We say $E(\mathcal{F})$ is a \textit{homogeneous Moran set determined by} $\mathcal{F}$, or it is a \textit{homogeneous Moran set determined by} $J$, $( n_k )$, $( c_k )$. We will need the following theorem of D. Feng, Z. Wen, and J. Wu from \cite{FengWenWu}.

\begin{thrm}[D. Feng, Z. Wen, and J. Wu]
If $S$ is a homogeneous Moran set determined by $J$, $(n_k )$, $( c_k )$, then
\begin{equation*}
\liminf_{k \to \infty} \frac{\log n_1 n_2 \cdots n_k}{-\log c_1 c_2 \cdots c_{k+1} n_{k+1}} \leq \dimh{S} \leq \liminf_{k \to \infty} \frac{\log n_1 n_2 \cdots n_k}{-\log c_1 c_2 \cdots c_k}.
\end{equation*}
\end{thrm}

Given basic sequences $\alpha = (\alpha_i)$ and  $\beta = (\beta_i)$, 
sequences of non-negative integers $s = (s_i), t = (t_i), \upsilon = (\upsilon_i),$ and $F = (F_i)$, and a sequence of sets $I = (I_i)$ such that $I_i \subseteq \{0, 1, \cdots, \beta_i-1\}$, define the set $\Theta(\alpha, \beta, s, t, \upsilon, F, I)$ as follows.
Let $Q=Q(\alpha, \beta, s, t, \upsilon)=(q_n)$ be the following basic sequence:
\begin{equation}
\left[[\alpha_1]^{s_1} [\beta_1]^{t_1}\right]^{\upsilon_1} \left[[\alpha_2]^{s_2} [\beta_2]^{t_2}\right]^{\upsilon_2} \left[[\alpha_3]^{s_3} [\beta_3]^{t_3}\right]^{\upsilon_3} \cdots.
\end{equation}
Define the function
$$
i(n) = \min \br{t : \sum_{i=1}^{t-1} \upsilon_i (s_i + t_i) < n}.
$$
Set 
$$
\Phi_\alpha(i,c,d) = \sum_{j =1}^{i-1} \upsilon_j s_j + c s_i + d
$$ 
where $0 \leq c < \upsilon_i$ and $0 \leq d < s_i$ and let the functions $i_\alpha(n)$, $c_\alpha(n)$, and $d_\alpha(n)$ be such that $\Phi^{-1}_\alpha(n) = (i_\alpha(n), c_\alpha(n), d_\alpha(n))$. Note this is possible since $\Phi_\alpha$ is a bijection from $\mathcal{U} = \br{(i,c,d) \in \mathbb{N}^3 : 0 \leq c < \upsilon_i, 0 \leq d < s_i}$ to $\mathbb{N}$. Define the functions
$$
G(n) = \sum_{j=1}^{i_\alpha(n)-1} \upsilon_j (s_j+t_j) + c_\alpha(n)\pr{s_{i_\alpha(n)} + t_{i_\alpha(n)}} + d_\alpha(n)
$$ 
and $g(n) = \min \br{t : G(t) \geq n}$. Note that $i_\alpha(g(n)) = i(n)$ and $c_\alpha(g(n)) = c_\alpha(n)$. Furthermore, define $C_\alpha(n) = \pr{\sum_{j=1}^{i_\alpha(n)-1} u_j} + c_\alpha(n)$.

We consider the condition on $n$
\begin{equation}\labeq{VNcond}
\pr{n-\sum_{j=1}^{i(n)-1} \upsilon_j(s_j+t_j)} \mod (s_{i(n)}+t_{i(n)}) \geq s_{i(n)}.
\end{equation}
Define the intervals 
$$
V(n) =
\begin{cases}
	I_{i(n)} & \text{ if condition \refeq{VNcond} holds}  \\
\ \\
	\left [F_{G(n)}, F_{G(n)}+1 \right ) & \text{ else}
\end{cases}.
$$
That is, we choose digits from $I_{i(n)}$ in positions corresponding to the bases obtained from the sequence $\beta$ and choose a specific digit from $F$ for the bases obtained from the sequence $\alpha$.
Set
$$
\Theta(\alpha, \beta, s, t, \upsilon, F, I) = \br{x = 0.E_1 E_2 \cdots \wrt{Q} : E_n \in V(n)}.
$$
We will need the following basic lemma to prove \refl{HDT} and elsewhere in this paper.
\begin{lem}\labl{tcorr}
Let $L$ be a real number and $(a_n)_{n=1}^\infty$ and $(b_n)_{n=1}^\infty$ be two sequences of positive real numbers such that
$$
\sum_{n=1}^{\infty} b_n=\infty \hbox{ and } \lim_{n \to \infty} \frac {a_n} {b_n}=L.
$$
Then
$$
\lim_{n \to \infty} \frac {a_1+a_2+\ldots+a_n} {b_1+b_2+\ldots+b_n}=L.
$$
\end{lem}
\begin{lem}\labl{HDT}
Given basic sequences $\alpha = (\alpha_i)$ and  $\beta = (\beta_i)$, 
sequences of non-negative integers  $s = (s_i), t = (t_i), \upsilon = (\upsilon_i),$ and $F = (F_i)$, and a sequence of sets $I = (I_i)$ such that $I_i \subseteq \{0, 1, \cdots, \beta_i -1\}$ such that the following conditions hold:

\begin{align}
\labeq{HDT1}
&\lim_{n \to \infty} \frac{s_n \log \alpha_n}{\sum_{i=1}^{n-1} \upsilon_i t_i \log \beta_i} = 0;\\
\labeq{HDT2}
&\lim_{n \to \infty} \frac{s_n \log \alpha_n}{t_n \log \beta_n} = 0. 
\end{align}
Then 
$$
\dimh{\Theta(\alpha, \beta, s, t, \upsilon, F, I)} = \gamma:= \lim_{n \to \infty} \frac{\log |I_n|}{\log \beta_n}.
$$
\end{lem}
\begin{proof}
Note that $\Theta(\alpha, \beta, s, t, \upsilon, F, I)$ is a homogeneous Moran set with 
$$
n_k = 
\begin{cases}
|I_k| & \text{if } q_k = \beta_{i(k)}\\
1  & \text{if } q_k = \alpha_{i(k)}
\end{cases}
$$
and $c_k = \frac{1}{q_k}$.
Thus
\begin{align*}
& \dimh{\Theta(\alpha, \beta, s, t, \upsilon, F, I)} \geq \liminf_{k \to \infty} \frac{\log n_1 n_2 \cdots n_k}{-\log c_1 c_2 \cdots c_{k+1} n_{k+1}} \\
& \geq \lim_{n \to \infty} \frac{\sum_{j=1}^{i(n)-1} \sum_{k=1}^{u_j} t_j  \log|I_i| + \sum_{j=1}^{b(n)}  t_{i(n)} \log|I_{i(n)}|}
{\sum_{j=1}^{i(n)-1} \sum_{k=1}^{u_j} \left [t_j \log \beta_j + s_j \log \alpha_j \right ] + \sum_{j=1}^{b(n)} \left [ t_{i(n)} \log \beta_{i(n)}+ s_{i(n)} \log \alpha_{i(n)}  \right ] + s_{i(n)} \log \alpha_{i(n)}} \\
& = \lim_{n \to \infty} \frac{\pr{\sum_{j=1}^{i(n)-1} u_j t_j \gamma \log \beta_j}  + b(n) t_{i(n)} \gamma \log \beta_{i(n)}}
{\sum_{j=1}^{i(n)-1} \sum_{k=1}^{u_j} \left [t_j \log \beta_j + s_j \log \alpha_j \right ] + \sum_{j=1}^{b(n)} \left [ t_{i(n)} \log \beta_{i(n)}+ s_{i(n)} \log \alpha_{i(n)} \right ] + s_{i(n)} \log \alpha_{i(n)}}\\
&\hbox{where we have used \refl{tcorr}.}\\
&=  \lim_{n \to \infty} \frac{\pr{\sum_{j=1}^{i(n)-1} u_j t_j \gamma \log \beta_j }  + b(n) t_{i(n)} \gamma \log \beta_{i(n)}}
{\pr{\sum_{j=1}^{i(n)-1} u_j t_j \log \beta_j }  + b(n) t_{i(n)} \log \beta_{i(n)}  + s_{i(n)} \log \alpha_{i(n)} }\\
&\hbox{which follows from \refeq{HDT2}.}\\
& = \lim_{n \to \infty} \frac{\pr{\sum_{j=1}^{i(n)-1} u_j t_j \gamma \log \beta_j } + b(n) t_{i(n)} \gamma \log \beta_{i(n)} }
{\pr{\sum_{j=1}^{i(n)-1} u_j t_j \log \beta_j} + b(n) t_{i(n)} \log \beta_{i(n)} } = \gamma.
\end{align*}
which we get from \refeq{HDT1}. The upper bound follows from a similar calculation.
\end{proof}

For a sequence of real numbers $X = (x_n)$ with $x_n \in [0,1)$ and an interval $I \subseteq [0,1]$,  define $A_n(I,X) = \# \{i\leq n: x_i \in I \}$.
We will need the following standard definition and lemma that we quote from \cite{KuN}.

\begin{definition}
Let $X = \pr{x_1, \cdots , x_N}$ be a finite sequence of real numbers. The number $$D_N = D_N(X) = \sup_{0 \leq \alpha \leq \beta \leq 1} \left | \frac{A_N([\alpha, \beta), X)}{N} - (\beta - \alpha) \right |$$ is called the {\it discrepancy} of the sequence $\omega$.
\end{definition}

It is well known that a sequence $X$ is uniformly distributed mod $1$ if and only if $D_N(X) \to 0$.

\begin{lem} \labl{DiscKuN}
Let $x_1, x_2, \cdots, x_N$ and $y_1, y_2, \cdots, y_N$ be two finite sequences in $[0,1)$.  Suppose $\epsilon_1, \epsilon_2, \cdots, \epsilon_N$ are non-negative numbers such that $|x_n-y_n| \leq \epsilon_n$ for $1 \leq n \leq N$.  Then, for any $\epsilon \geq 0$, we have
$$
|D_N(x_1,\cdots,x_N)-D_N(y_1,\cdots,y_N)| \leq 2\epsilon+\frac {\overline{N}(\epsilon)}{N},
$$
where $\overline{N}(\epsilon)$ denotes the number of $n$, $1 \leq n \leq N$, such that $\epsilon_n>\epsilon$.
\end{lem}

\section{Results}

We will compute the Hausdorff dimension of   difference sets formed by taking unions or intersections of the sets $\NQ$, $\RNQ$, and $\DNQ$.  Every other similar result will follow as a corollary of one of these theorems, by using similar techniques, or by \reff{figure1}.
\begin{proof}[Proof of \reft{RDN}]
Let $P = (p_i)$ with $p_i = \floor{\log i}+2$ and $\xi \in \N{P}$ with $\xi = .F_1 F_2\cdots \wrt{P}$. Fix a sequence $X = (x_n)$ that is uniformly distributed modulo 1. Define the sequences
\begin{align*}
&\nu_n = \min\br{t : \frac{\sum_{i=0}^{n-1} \log q_{I(n-1)+i}}{\sum_{i=0}^{j-I(n-1)-1} \log q_{I(n-1)+i}} < \frac{1}{n}, \forall j\geq t};\\
&\upsilon_{n,k} = \min \br{t : \frac{Q_n^{(k)}}{\sum_{i=1}^j P_{i-k+1}^{(k)} } < \frac{1}{n}, \forall j \geq t }; \\
&L_0 = 0; \\
&L_n = \max \br{ \min\br{t : \log(q_j) > n, \forall j \geq t}, L_{n-1} + n^2, L_{n-1}+\nu_n, \max_{k \leq n} \br{\upsilon_{n,k}}}
\end{align*}
and set $i(n) = \max\br{j : L_j \leq n}$.
Note that $\nu_n$ and $\upsilon_{n,k}$ exist since $Q$ is infinite in limit and $P$ is fully divergent.
Define the set 
$$
S = \bigcup_{n=1}^\infty \{L_n, L_n+1, \cdots, L_n+n-1\}.
$$
Note that this set has density $0$ since 
$$
\frac{\sum_{i=1}^n i}{\sum_{i=1}^n i + t_i} \leq \frac{\sum_{i=1}^n i}{\sum_{i=1}^n i+i^2} \to 0 \hbox{ as $n$ goes to infinity.}
$$
Define the intervals
$$
V(n) = 
\begin{cases}
[F_{n-L_i}, F_{n-L_i}+1) &\text{if } n \in [L_i, L_i+1, \cdots, L_i+i]\\
[x_n q_n - \omega_n, x_n q_n + \omega_n) \cap [\ceil{ \log i(n)}, q_{n}-1] & \text{else}  
\end{cases}
$$
where
$$
\omega_n = q_n^{1-\epsilon_i} \hbox{ and } \epsilon_i = \frac{\min\br{\log q_1 \cdots q_{i-1}, \log q_i }^{1/2}}{\log q_i}
$$
Set
$
\Lambda_Q = \{ x = .E_1 E_2 \wrt{Q} : E_n \in V(n)\}.
$
We claim that $\Lambda_Q \subseteq \RDN$ and $\dimh{\Lambda_Q} = 1$.
Let $x \in \Lambda_Q$ and let $B$ be a block of length $k$. 
Note that by the definition of $L_n$, there are only finitely many values $n \in \mathbb{N}\backslash S$ such that $B$ occurs at position $n$ in the $Q$-Cantor series expansion of $x$. 
This is because all digits $E_n$ with $n \in \mathbb{N} \backslash S$ must be greater than $\ceil{\log i(n)}$ by the definition of $V(n)$ and since $i(n)$ tends to infinity as $n$ does. 
Thus, if $m$ is the maximum digit for the block $B$, we have that for $n \in \mathbb{N} \backslash S$ with $i(n) > m$, that $E_n > m$. Thus 
$
N_n^Q(B,x) = \sum_{i=1}^{i(n)} N_{i-k+1}^P(B, \xi) + O(1).
$
So for any two blocks $B_1$ and $B_2$ of length $k$, we have
\begin{align*}
\lim_{n \to \infty} \frac{N_n^Q(B_1,x)}{N_n^Q(B_2, x)} &= \lim_{n \to \infty} \frac{\sum_{i=1}^{i(n)} N_{i-k+1}^P(B_1, \xi) + O(1)}{\sum_{i=1}^{i(n)} N_{i-k+1}^P(B_2, \xi) + O(1)} \\
&= \lim_{n \to \infty} \frac{N_{n-k+1}^P(B_1, \xi)}{N_{n-k+1}^P(B_2, \xi)} = 1.
\end{align*}
Thus $x \in \RN{Q}$.

Consider the sequence $Y = \pr{\frac{E_n}{q_n}}$. For $n \in \mathbb{N} \backslash S$, we have $\left |\frac{E_n}{q_n} - x_n \right | < \frac{\omega_n}{q_n}$, which tends to $0$ as $n$ goes to infinity. We therefore have for $\epsilon>0$ that $\overline{N}(\epsilon) = O(1) + \# S \cap \{1, \cdots, N\}$. Thus by \refl{DiscKuN}
$$
\left |D_N(X) - D_N(Y) \right | < 2 \epsilon + \frac{O(1)}{N} + \frac{\# S \cap \{1, \cdots, N\}}{N} \to 2 \epsilon
$$
as $N$ tends to infinity. Since the inequality holds for all $\epsilon>0$, we have that $\pr{\frac{E_n}{q_n}}$ is uniformly distributed mod 1. Thus $x \in \DN{Q}$.

Note that
$$
\lim_{n \to \infty} \frac{N_n^Q(B,x)}{\sum_{i=1}^{i(n)} P_{i-k+1}^{(k)} } = 1.
$$
However, 
$$
\lim_{n \to \infty} \frac{Q_n^{(k)}}{\sum_{i=1}^{i(n)} P_{i-k+1}^{(k)} } = 0
$$
by the definition of $L_n$, so $x \not \in \N{Q}$.
Thus $\Lambda_Q \subseteq \RDN$.

Evidently $\Lambda_Q$ is a homogeneous Moran set with $n_k =|V(k)|$ and $c_k = \frac{1}{q_k}$. Thus
\begin{align*}
\dimh{\Lambda_Q} &\geq \liminf_{k \to \infty} \frac{\log n_1 \cdots n_k}{- \log c_1 \cdots c_{k+1} n_{k+1}} \\
& = \liminf_{n \to \infty} \frac{\sum_{i=1}^k \chi_{\mathbb{N}\backslash S}(i) \pr{1- \epsilon_i} \log q_i}{\sum_{i=1}^k \log q_i + \log q_{k+1}}\\
&= \liminf_{n \to \infty} \pr{1 - \frac{\sum_{j=1}^{i(n)} \sum_{k=0}^{j-1} \log q_{L_j+k}}{\sum_{j=1}^{i(n)} \sum_{k=0}^{i(j)-i(j-1)} \log q_{L_j+k}} }\\
&= \liminf_{n \to \infty} \pr{1 - \frac{\sum_{i=0}^{n-1} \log q_{L_n+i}}{\sum_{i=0}^{L_n-L_{n-1}} \log q_{L_n+i}}}= 1
\end{align*}
by the definition of $L_n$. Thus 
$$
\dimh{\Lambda_Q} =1 \hbox{ and } \dimh{\RDN}=~1.
$$
\end{proof}

\begin{cor}
If $Q$ is infinite in limit, then $\dimh{\RNQ}=1$.
\end{cor}

\begin{thrm}
If $Q$ is infinite in limit, then 
$$
\dimh{\RN{Q} \backslash \pr{\bigcup_{j=1}^\infty \Nk{Q}{j} \cup \DN{Q}}} =1.
$$
\end{thrm}
\begin{proof}
The proof is the same as \reft{RDN}, but with $X = (x_n)$ a sequence that is not uniformly distributed mod 1.
\end{proof}
\begin{thrm}\labt{DNRN}
If $Q$ is infinite in limit, then
$$
\dimh{\DNQ \backslash \bigcup_{j=1}^\infty \RNk{Q}{j}}=1.
$$
\end{thrm}
\begin{proof}
The proof is the same as \reft{RDN}, but we choose $\xi = .E_1 E_2 \cdots \wrt{P}$ such that the digit $0$ never occurs.
\end{proof}
We will need to refer to the following four conditions.
\begin{align} \labeq{NQ}
&\lim_{n \to \infty} \frac{t_n \alpha_n^k}{s_n \beta_n^k} = 0;\\
\labeq{NotNQ}
&\lim_{n \to \infty} \frac{t_n \alpha_n^k}{s_n \beta_n^k} > 0;\\
\labeq{RNQ}
&\lim_{n \to \infty} \frac{\alpha_n^k}{s_n} = 0;\\
\labeq{DNQ}
&\lim_{n \to \infty} \frac{\sum_{i=1}^n \upsilon_i s_i }{\sum_{i=1}^n \upsilon_i (s_i+t_i)} = 0.
\end{align}

\begin{thrm}\labt{NDN}
Suppose that $Q=Q(\alpha, \beta, s, t, \upsilon)$ is infinite in limit,  $k$-divergent (resp. fully divergent), and satisfies conditions \refeq{HDT1}, \refeq{HDT2}, \refeq{NQ} for all $k$, \refeq{RNQ}, and \refeq{DNQ}.  If $\alpha_i=o(\beta_i)$, then 
$$
\dimh{\bigcap_{j=1}^k \Nk{Q}{j} \backslash \DNQ}=1 \pr{ \hbox{resp. } \dimh{\NQ \backslash \DNQ}=1 }.
$$
\end{thrm}

\begin{proof}
We will prove the statement for when $Q=(q_n)$ is fully divergent. The proof for when $Q$ is $k$-divergent follows similarly.
Define the basic sequence $P$ by
$$
P=[\alpha_1]^{s_1\upsilon_1}[\alpha_2]^{s_2\upsilon_2}[\alpha_3]^{s_3\upsilon_3}[\alpha_4]^{s_4\upsilon_4}\cdots.
$$
We note that $P$ is fully divergent since $Q$ is fully divergent.
By \reft{measure}, there exists a real number $\xi=E_0.E_1E_2\cdots\wrt{P}$ that is an element of $\N{P}$.  Set 
$$
I_i = \br{ \alpha_i, \alpha_i+1, \cdots, \floor{\beta_i^{1-(1/\log \beta_i)^{1/2}}} + 1}
$$
and $F_i=E_i$. Note that $\lim_{n \to \infty} \log |I_n|/\log \beta_n=1$, so $\dimh{\Theta(\alpha, \beta, s, t, \upsilon, F, I)}=1$ by \refl{HDT}.  We now wish to show that
$$
\Theta(\alpha, \beta, s, t, \upsilon, F, I) \subseteq  \NQ \backslash \DNQ.
$$
Let $k$ and $n$ be natural numbers, $B$ be a block of length $k$, and $x \in \Theta(\alpha, \beta, s, t, 
\upsilon, F, I)$.  We wish to show that 
$$
N_{g(n)}^P(B, \xi) - k C_\alpha(g(n)) \leq N_n^Q(B,x) \leq N_{g(n)}^P(B,\xi)+O(1).
$$
Let $m$ be the 
maximum digit in the block $B$.  Since $\min I_i \to \infty$, we know that there are only finitely many indices $i$ such that 
$m> \min I_i$.  Thus, there are at most finitely many occurrences of $B$ starting at position $n$ when 
$q_n=\beta_{i(n)}$. If every occurrence of $B$ in $\xi$ occurs at the corresponding place in $x$, then we have 
$$
N_{g(n)}^P(B, \xi)  + O(1) = N_n^Q(B,x).
$$ 
If some of the occurrences of $B$ in $\xi$ do not occur in the corresponding places in $x$, then we have $N_n^Q(B,x) \leq N_{g(n)}^P(B, \xi)$.

On the other hand, the total number of places up to position $n$ where $B$ can occur in the $P$-Cantor series expansion of $\xi$ but $B$ does not occur in the corresponding positions in the $Q$-Cantor series expansion of $x$ is at most $k C_\alpha(n)$, the total length of the  last $k$ terms of the substrings $[\alpha_i]^{s_i}$ of $P$. Thus 
$$
N_{g(n)}^P(B,\xi) - k C_\alpha(g(n)) \leq N_n^Q(B, x) \leq N_{g(n)}^P (B, \xi) + O(1).
$$

Many of the following calculations use \refl{tcorr}. Note that 
$$
\pnk = \sum_{j = 1}^{i_\alpha(n)-1} \frac{s_j \upsilon_j }{\alpha_j^k} + \frac{s_{i(n)} b_\alpha(n)}{\alpha_{i_\alpha(n)}^k}
$$
and 
\begin{align*}
\qnk &= \pr{\sum_{j =1}^{i(n)-1} \frac{(s_j - k) \upsilon_j}{\alpha_j^k} +  \frac{(t_j-k)\upsilon_j}{\beta_j^k} + \pr{\sum_{l=1}^{k-1} \frac{\upsilon_j}{\beta_j^l \alpha_j^{k-l}} + \frac{\upsilon_j}{\alpha_j^l \beta_j^{k-l}}  }} \\
&+ \frac{c(n) (s_{i(n)} - k)}{\alpha_{i(n)}^k} + \frac{c(n)(t_{i(n)} - k)}{\beta_{i(n)}^k} + \pr{\sum_{l=1}^{k-1} \frac{\upsilon_{i(n)}}{\beta_{i(n)}^l \alpha_{i(n)}^{k-l}} + \frac{\upsilon_j}{\alpha_{i(n)}^l \beta_{i(n)}^{k-l}}}.
\end{align*}

Note that by \refeq{HDT1} and \refeq{HDT2}, we have that
$$
\lim_{n \to \infty} \frac{\qnk}{\pr{\sum_{j =1}^{i(n)-1} \frac{(s_j-k) \upsilon_j}{\alpha_j^k} +  \frac{(t_j-k)\upsilon_j}{\beta_j^k}} + \frac{c(n) (s_{i(n)} - k)}{\alpha_{i(n)}^k}+ \frac{c(n)(t_{i(n)} - k)}{\beta_{i(n)}^k}} = 1.
$$

Thus
\begin{align*}
\lim_{ n \to \infty} \frac{\qnk}{P_{g(n)}^{(k)}} &
= \lim_{n \to \infty} \frac{\pr{\sum_{j =1}^{i(n)-1} \frac{(s_j-k) \upsilon_j}{\alpha_j^k} +  \frac{(t_j-k)\upsilon_j}{\beta_j^k}} + \frac{c(n) (s_{i(n)} - k)}{\alpha_{i(n)}^k}+ \frac{c(n)(t_{i(n)} - k)}{\beta_{i(n)}^k}}{\pr{\sum_{j=1}^{i(n)-1} \frac{s_j \upsilon_j}{\alpha_j^k}} + \frac{c(n) s_{i(n)}}{\alpha_{i(n)}^k}} \\
& = \lim_{n \to \infty} \frac{s_n - k}{s_n}  + \frac{(t_n-k) \alpha_n^k}{s_n \beta_n^k} = 1 + \lim_{n \to \infty} \frac{t_n \alpha_n^k}{s_n \beta_n^k}=1.
\end{align*}
Furthermore, we have that
\begin{align*}
\lim_{n \to \infty} \frac{C_\alpha(g(n))}{P_{g(n)}^{(k)}} &= \lim_{n \to \infty} \frac{\pr{\sum_{j=1}^{i(n)-1} \upsilon_j } + c(n)}{\pr{\sum_{j=1}^{i(n)-1} \frac{s_j \upsilon_j -k}{\alpha_j^k} } + \frac{c(n) s_{i(n)} - k}{\alpha_{i(n)}^k}   } \\
&= \lim_{n \to \infty} \frac{\alpha_n^k }{s_n - k/\upsilon_n} = \lim_{n \to \infty} \frac{\alpha_n^k}{s_n} = 0.
\end{align*}
Since $\xi \in \N{P}$, we have that
$$\lim_{n \to \infty} \frac{N_n^Q(B,x)}{\qnk} = \lim_{n \to \infty} \frac{N_n^Q(B,x)}{P_{g(n)}^{(k)}} =1.$$
Therefore, $x \in \N{Q}$.

For $n$ where $q_n = \beta_{i(n)}$, we have 
\begin{equation}\labeq{whereblahgoestozero}
\frac{E_n}{q_n} \leq \frac{\beta_{i(n)}^{1-\log^{-1/2} \beta_{i(n)}}}{\beta_{i(n)}} \to 0 \hbox{ as } n \to \infty.
\end{equation}
Up to position $n$ there are at least $\sum_{j=1}^{i(n)} \upsilon_i t_i + c(n) t_{i(n)}$ such places where \refeq{whereblahgoestozero} holds. By \refeq{DNQ}, we have
$$
\lim_{n \to \infty} \frac{\sum_{j=1}^{i(n)} \upsilon_i t_i + c(n) t_{i(n)}}{n} = 1
$$
so the sequence $\pr{\frac{E_n}{q_n}}$ is not uniformly distributed mod 1. Thus $x \notin \DN{Q}$ and $\Theta(\alpha, \beta, s, t, \upsilon, F, I) \subseteq \N{Q} \backslash \DN{Q}$, which implies that $\dimh{\N{Q} \backslash \DN{Q}} =1.$
\end{proof}

\begin{thrm}\labt{NDNk}
Suppose that $Q=Q(\alpha, \beta, s, t, \upsilon)$ is infinite in limit, fully divergent, and satisfies conditions \refeq{HDT1}, \refeq{HDT2}, \refeq{NQ} for $k\geq \ell$, \refeq{NotNQ} for $\ell < k$, and \refeq{RNQ}. Then 
$$
\dimh{\bigcap_{j=\ell}^\infty \Nk{Q}{j} \Big \backslash \bigcup_{j=1}^{\ell-1} \Nk{Q}{j}}=1.
$$
\end{thrm}

\begin{proof}
Define the same basic sequence $P$ and sequences $I$ and $F$ as in the proof of \reft{NDN}.
The  same arguments regarding the asymptotics of $N_n^Q(B,x)$ for $x \in \Theta(\alpha, \beta, s, t, F, I)$ hold,  so
$$
\lim_{n \to \infty} \frac{N_{n}^Q(B,x)}{P_{g(n)}^{(k)}} = 1.
$$
But since \refeq{NQ} holds for $k \geq \ell$, we have that
$$
\lim_{n \to \infty} \frac{Q_n^{(k)}}{P_{g(n)}^{(k)}} = 1 + \lim_{n \to \infty} \frac{t_n \alpha_n^k}{s_n \beta_n^k} = 1.
$$
Thus $x$ is $Q$-normal of orders greater than or equal to $\ell$.
\end{proof}

\begin{example}
Set $\alpha_n =  \floor{\log \log(n+2)}+2$, $\beta_n = \floor{\log n}+2$, $s_n = \floor{\log n}$, $t_n = n$, and $\upsilon_n = 2^n$. Then the conditions of \reft{NDN} are satisfied.
\end{example}

\begin{example}
Fix some integer $\ell$. Set $\alpha_n = \floor{\log \log( n+2)}+2$, $\beta_n = \floor{\log n}+2$, $s_n = \floor{\log n}$, $t_n = \floor{\pr{\frac{\beta_n}{\alpha_n}}^{\ell+1} s_n}$, and $\upsilon_n = 2^n$. Then  the conditions of \reft{NDNk} are satisfied.
\end{example}

\begin{proof}[Proof of \reft{digitrange}]
Let $\gamma=\dimm{S}$, $\alpha_i = 2$, $\beta_i = q_i$, $s_i = 0$, $t_i = 1$, $\upsilon_i = 1$, $F_i = 0$, and 
$$
I_i = S \cap \{0, \cdots, q_i -2 \}.
$$ 
Then \refeq{HDT1} and \refeq{HDT2} clearly hold. Note that 
$$
\WQS \subseteq \Theta(\alpha, \beta, s, t, \upsilon, F, I),
$$ 
 so $\dimh{\WQS} \leq \gamma$.

To get a lower bound, we construct a subset of $\WQS$ with Hausdorff dimension $\gamma$. To do this, let $T \subset \mathbb{N}$ be an infinite set that is sparse enough such that 
$$
\lim_{k \to \infty} \frac{\sum_{i=1}^k \chi_T(i) \log \# (S \cap \{0, \cdots , q_i-2 \})}{\sum_{i=1}^k \log \# (S \cap \{0, \cdots, q_i -2\})} = 0.
$$
Note that such a $T$ exists since $\lim_{k\to \infty} \sum_{i=1}^k \log \# (S \cap \{0, \cdots, q_i-2\}) = \infty$.

Let $f: T \to S$ be a surjective function such that for all $t \in T$, we have $q_t > f(t)$. Such an $f$ exists since $\min S < \min Q$, $T$ is infinite, and $Q$ is infinite in limit. Consider the homogeneous Moran set $C$ with 
$$
n_k =
\begin{cases}
1 & \text{ if } k \in T\\
\# S \cap \{0, \cdots, q_k -2 \} & \text{ else}
\end{cases}
$$
and $c_k = \frac{1}{q_k}$ described as follows:
If $k \in T$, then for any $x \in C$, $E_k(x) = f(k)$. Otherwise, $E_k(x) \in S \cap \{0, \cdots, q_k-2 \}$.
Since $f$ is surjective, we have that for any $x \in C$ that $\SQx = S$,  so $C \subseteq \WQS$.
But 
\begin{align*}
\dimh{C} &\geq \liminf_{k \to \infty} \frac{\log n_1 \cdots n_k}{- \log c_1 \cdots c_{k+1} n_{k+1}} \\
& = \lim_{k \to \infty} \frac{\sum_{i=1}^k \chi_{\mathbb{N} \backslash T}(i) \log\#( S \cap \{0, \cdots, q_i-2\})}{\sum_{i=1}^k \log q_i + \log q_{k+1}}\\
& = \lim_{k \to \infty} \frac{\sum_{i=1}^k \log\# (S \cap \{0, \cdots, q_i-2\})}{\sum_{i=1}^k \log q_i}\\
& = \lim_{k \to \infty} \frac{\log\# (S \cap \{0, \cdots, q_k-2\})}{\log q_k} = \gamma.
\end{align*}
Thus $\dimh{\WQS} \geq \gamma$,  so we have $\dimh{\WQS} = \gamma$.

\end{proof}

\section{Further problems}
\begin{problem}
For which irrational $x$ does there exist a basic sequence $Q$ where $x \in \RDN$.  The same question may be asked about several of the other sets discussed in this paper.  We remark that it is already known that for every irrational $x$ there exist uncountably many basic sequences $Q$ where $x \in \DNQ$.  See \cite{Laffer}.
\end{problem}

\begin{problem}
Prove that the conclusions of \reft{NDN} and \reft{NDNk} hold for all $Q$ that are infinite in limit and fully divergent.
\end{problem}

\begin{problem}
In \cite{Mance7} sufficient conditions are given under countable  intersections of sets of the form $\DNQ \backslash \bigcup_{j=1}^\infty \RNk{Q}{j}$ have full Hausdorff dimension.  Surely a similar result holds for many of the sets described in this paper.  Necessary and sufficient conditions similar to conditions found in the paper of W. M. Schmidt \cite{SchmidtRelated} may be possible.
\end{problem}

\bibliographystyle{amsplain}

\input{HDDifferenceSets.bbl}

\end{document}

%% file: HDDifferenceSets.bbl
\providecommand{\bysame}{\leavevmode\hbox to3em{\hrulefill}\thinspace}
\providecommand{\MR}{\relax\ifhmode\unskip\space\fi MR }
\providecommand{\MRhref}[2]{%
  \href{http://www.ams.org/mathscinet-getitem?mr=#1}{#2}
}
\providecommand{\href}[2]{#2}